\documentclass[11pt]{amsart}
\usepackage[margin=30mm]{geometry}
\usepackage{amsmath,amssymb}
\usepackage{amsthm}

\newtheorem{thm}{Theorem}[section]

\newtheorem{lem}{Lemma}[section]
\newtheorem{cor}{Corollary}[section]
\newtheorem{defi}{Definition}[section]
\newtheorem{ex}{Example}[section]
\newtheorem{rem}{Remark}[section]
\newtheorem*{ques}{\it{Question}}

\begin{document}

\title{Some results on shadowing and local entropy properties of dynamical systems}
\author{Noriaki Kawaguchi}
\address{Research Institute of Science and Technology, Tokai University, 4-1-1 Kitakaname, Hiratsuka, Kanagawa 259-1292, Japan}
\email{gknoriaki@gmail.com}

\begin{abstract}
We consider some local entropy properties of dynamical systems under the \linebreak assumption of shadowing. In the first part, we give necessary and sufficient conditions for shadowable points to be certain entropy points. In the second part, we give some necessary and sufficient conditions for (non) h-expansiveness under the assumption of shadowing and chain transitivity; and use the result to present a counter-example for a question raised by Artigue et al. [Proc. Amer. Math. Soc. 150 (2022), 3369--3378].
\end{abstract}

\maketitle

\markboth{NORIAKI KAWAGUCHI}{Some results on shadowing and local entropy properties of dynamical systems}

\section{Introduction}

{\em Shadowing}, introduced by Anosov and Bowen \cite{A, B2}, is a feature of hyperbolic dynamical systems and has played an important role in the global theory of dynamical systems (see \cite{AH} or \cite{P} for background). It generally refers to a phenomenon in which coarse orbits, or {\em pseudo-orbits}, are approximated by true orbits. In \cite{M}, by splitting the global shadowing into pointwise shadowings, Morales introduced the notion of {\em shadowable points}, which gives a tool for a local description of the shadowing phenomena. The study of shadowable points has been extended to shadowable points for flows \cite{AV}; pointwise stability and persistence \cite{DLM, JLM, KD2, KDD, KLM}; shadowable measures \cite{S}; average shadowable and specification points \cite{DKD, KD1}; eventually shadowable points \cite{DJM}; shadowable points of set-valued dynamical systems \cite{LNY}; and so on.

In \cite{K2}, some sufficient conditions are given for a shadowable point to be an entropy point. Recall that the notion of {\em entropy points} is obtained by a concentration of positive topological entropy at a point \cite{YZ}. Also, in \cite{AR}, the notion of shadowable points is applied to obtain pointwise sufficient conditions for positive topological entropy (see also \cite{RA}). In the first part of this paper, we improve the result of \cite{K2} by giving necessary and sufficient conditions for shadowable points to be certain entropy points. The {\em h-expansiveness} is another local entropy property of dynamical systems \cite{B1}. In the second part of this paper, we present several necessary and sufficient conditions for (non) h-expansiveness under the assumption of shadowing and chain transitivity; and use the result to obtain a counter-example for a question in \cite{ACCV2}.

We begin with a definition. Throughout, $X$ denotes a compact metric space endowed with a metric $d$.

\begin{defi}
\normalfont
Given a continuous map $f\colon X\to X$ and $\delta>0$, a finite sequence $(x_i)_{i=0}^{k}$ of points in $X$, where $k>0$ is a positive integer, is called a {\em $\delta$-chain} of $f$ if $d(f(x_i),x_{i+1})\le\delta$ for every $0\le i\le k-1$.  A $\delta$-chain $(x_i)_{i=0}^{k}$ of $f$ with $x_0=x_k$ is said to be a {\em $\delta$-cycle} of $f$. 
\end{defi}

Let $f\colon X\to X$ be a continuous map. For any $x,y\in X$ and $\delta>0$, the notation $x\rightarrow_\delta y$ means that there is a $\delta$-chain $(x_i)_{i=0}^k$ of $f$ with $x_0=x$ and $x_k=y$. We write $x\rightarrow y$ if $x\rightarrow_\delta y$ for all $\delta>0$. We say that $x\in X$ is a {\em chain recurrent point} for $f$ if $x\rightarrow x$, or equivalently, for any $\delta>0$, there is a $\delta$-cycle $(x_i)_{i=0}^{k}$ of $f$ with $x_0=x_k=x$. Let $CR(f)$ denote the set of chain recurrent points for $f$. We define a relation $\leftrightarrow$ in
\[
CR(f)^2=CR(f)\times CR(f)
\]
by: for any $x,y\in CR(f)$, $x\leftrightarrow y$ if and only if $x\rightarrow y$ and $y\rightarrow x$. Note that $\leftrightarrow$ is a closed equivalence relation in $CR(f)^2$ and satisfies $x\leftrightarrow f(x)$ for all $x\in CR(f)$. An equivalence class $C$ of $\leftrightarrow$ is called a {\em chain component} for $f$. We denote by $\mathcal{C}(f)$ the set of chain components for $f$.

A subset $S$ of $X$ is said to be $f$-invariant if $f(S)\subset S$. For an $f$-invariant subset $S$ of $X$, we say that $f|_S\colon S\to S$ is {\em chain transitive} if for any $x,y\in S$ and $\delta>0$, there is a $\delta$-chain $(x_i)_{i=0}^k$ of $f|_S$ with $x_0=x$ and $x_k=y$.

\begin{rem}
\normalfont
The following properties hold
\begin{itemize}
\item $CR(f)=\bigsqcup_{C\in\mathcal{C}(f)}C$,
\item Every $C\in\mathcal{C}(f)$ is a closed $f$-invariant subset of $CR(f)$,
\item $f|_C\colon C\to C$ is chain transitive for all $C\in\mathcal{C}(f)$,
\item For any $f$-invariant subset $S$ of $X$, if $f|_S\colon S\to S$ is chain transitive, then $S\subset C$ for some $C\in\mathcal{C}(f)$.
\end{itemize}
\end{rem}

Let $f\colon X\to X$ be a continuous map. For $x\in X$, we define a subset $C(x)$ of $X$ by
\[
C(x)=\{x\}\cup\{y\in X\colon x\rightarrow y\}.
\] 
By this definition, we easily see that for any $x\in X$, $C(x)$ is a closed $f$-invariant subset of $X$. We say that a closed $f$-invariant subset $S$ of $X$ is {\em chain stable} if for any $\epsilon>0$, there is $\delta>0$ for which every $\delta$-chain $(x_i)_{i=0}^k$ of $f$ with $x_0\in S$ satisfies $d(x_i,S)\le\epsilon$ for all $0\le i\le k$. A proof of the following lemma is given in Section 3.

\begin{lem}
$C(x)$ is chain stable for all $x\in X$.
\end{lem}

\begin{rem}
\normalfont
For any $x\in X$, since $C(x)$ is chain stable, it satisfies the following properties
\begin{itemize} 
\item $CR(f|_{C(x)})=C(x)\cap CR(f)$,
\item for every $C\in\mathcal{C}(f)$, $C\subset C(x)$ if and only if $C\cap C(x)\ne\emptyset$,
\item $\mathcal{C}(f|_{C(x)})=\{C\in\mathcal{C}(f)\colon C\subset C(x)\}$.
\end{itemize}
\end{rem}

Let $f\colon X\to X$ be a continuous map and let $\xi=(x_i)_{i\ge0}$ be a sequence of points in $X$. For $\delta>0$, $\xi$ is called a {\em $\delta$-pseudo orbit} of $f$ if $d(f(x_i),x_{i+1})\le\delta$ for all $i\ge0$. For $\epsilon>0$, $\xi$ is said to be {\em $\epsilon$-shadowed} by $x\in X$ if $d(f^i(x),x_i)\leq \epsilon$ for all $i\ge 0$.

\begin{defi}
\normalfont
Given a continuous map $f\colon X\to X$, $x\in X$ is called a {\em shadowable point} for $f$ if for any $\epsilon>0$, there is $\delta>0$ such that every $\delta$-pseudo orbit $(x_i)_{i\ge0}$ of $f$ with $x_0=x$ is $\epsilon$-shadowed by some $y\in X$. We denote by $Sh(f)$ the set of shadowable points for $f$.
\end{defi}

For a continuous map $f\colon X\to X$ and a subset $S$ of $X$, we say that $f$ has the {\em shadowing on $S$} if for any $\epsilon>0$, there is $\delta>0$ such that every $\delta$-pseudo orbit $(x_i)_{i\ge0}$ of $f$ with $x_i\in S$ for all $i\ge0$ is $\epsilon$-shadowed by some $y\in X$. We say that $f$ has the {\em shadowing property} if $f$ has the shadowing on $X$.

The next lemma is a basis for the formulation of Theorems 1.1 and 1.2.

\begin{lem}
For a continuous map $f\colon X\to X$ and $x\in X$, the following conditions are equivalent
\begin{itemize}
\item[(1)] $x\in Sh(f)$,
\item[(2)] $C(x)\subset Sh(f)$,
\item[(3)] $f$ has the shadowing on $C(x)$. 
\end{itemize}
\end{lem}

Next, we recall the definition of entropy points from \cite{YZ}. Let $f\colon X\to X$ be a continuous map. For $n\ge1$, the metric $d_n$ on $X$ is defined by
\[
d_n(x,y)=\max_{0\le i\le n-1}d(f^i(x),f^i(y))
\]
for all $x,y\in X$. For $n\ge1$ and $r>0$, a subset $E$ of $X$ is said to be {\em $(n,r)$-separated} if $d_n(x,y)>r$ for all $x,y\in E$ with $x\ne y$. Let $K$ be a subset of $X$. For $n\ge1$ and $r>0$, let $s_n(f,K,r)$ denote the largest cardinality of an $(n,r)$-separated subset of $K$. We define $h(f,K,r)$ and $h(f,K)$ by
\[
h(f,K,r)=\limsup_{n\to\infty}\frac{1}{n}\log{s_n(f,K,r)}
\]
and
\[
h(f,K)=\lim_{r\to0}h(f,K,r).
\]
We also define the topological entropy $h_{\rm top}(f)$ of $f$ by $h_{\rm top}(f)=h(f,X)$.
 
\begin{defi}
\normalfont
Let $f\colon X\to X$ be a continuous map. For $x\in X$, we denote by $\mathcal{K}(x)$ the set of closed neighborhoods of $x$.
\begin{itemize}
\item[(1)] $Ent(f)$ is the set of $x\in X$ such that $h(f,K)>0$ for all $K\in\mathcal{K}(x)$,
\item[(2)] For $r>0$, $Ent_r(f)$ is the set of $x\in X$ such that $h(f,K,r)>0$ for all $K\in\mathcal{K}(x)$,
\item[(3)] For $r>0$ and $b>0$, $Ent_{r,b}(f)$ is the set of $x\in X$ such that $h(f,K,r)\ge b$ for all $K\in\mathcal{K}(x)$. 
\end{itemize}
\end{defi}

\begin{rem}
\normalfont
The following properties hold
\begin{itemize}
\item $Ent(f)$, $Ent_r(f)$, $r>0$, and $Ent_{r,b}(f)$, $r,b>0$, are closed $f$-invariant subsets of $X$,
\item
\[
Ent(f)\subset Ent_r(f)\subset Ent_{r,b}(f)
\]
for all $r,b>0$,
\item for any closed subset $K$ of $X$ and $r>0$, if $h(f,K,r)>0$, then $K\cap Ent_r(f)\ne\emptyset$,
\item for any closed subset $K$ of $X$ and $r,b>0$, if $h(f,K,r)\ge b$, then $K\cap Ent_{r,b}(f)\ne\emptyset$.
\end{itemize}
\end{rem}

For a continuous map $f\colon X\to X$, we define $\mathcal{C}_o(f)$ as the set of $C\in\mathcal{C}(f)$ such that $C$ is a periodic orbit or an odometer. We also define $\mathcal{C}_{no}(f)$ to be
\[
\mathcal{C}_{no}(f)=\mathcal{C}(f)\setminus\mathcal{C}_o(f).
\]
We refer to Section 2.1 for various matters related to this definition.

The first theorem characterizes the shadowable points that are entropy points of a certain type.
  
\begin{thm}
Given a continuous map $f\colon X\to X$ and $x\in Sh(f)$,
\[
x\in\bigcup_{r>0}Ent_r(f)
\]
if and only if $\mathcal{C}_{no}(f|_{C(x)})\ne\emptyset$.
\end{thm}

The following theorem characterizes the shadowable points that are entropy points of other type.

\begin{thm}
Given a continuous map $f\colon X\to X$ and $x\in Sh(f)$,
\[
x\in\bigcup_{r,b>0}Ent_{r,b}(f)
\]
if and only if $h_{\rm top}(f|_{C(x)})=h(f,C(x))>0$.
\end{thm}

\begin{rem}
\normalfont
In \cite{YZ}, a point of
\[
\bigcup_{r,b>0}Ent_{r,b}(f)
\]
is called a {\em uniform entropy point} for $f$.
\end{rem}

\begin{rem}
\normalfont
In Section 3, we give an example of a continuous map $f\colon X\to X$ such that
\begin{itemize}
\item $f$ has the shadowing property and so satisfies $X=Sh(f)$,
\item
\[
Ent(f)\setminus\bigcup_{r>0}Ent_r(f)
\]
is a non-empty set.
\end{itemize}
\end{rem}

Before we state the next theorem, we introduce some definitions.

\begin{defi}
\normalfont
For a continuous map $f\colon X\to X$ and $(x,y)\in X^2$,
\begin{itemize}
\item $(x,y)$ is called a {\em distal pair} for $f$ if
\[
\liminf_{i\to\infty}d(f^i(x),f^i(y))>0,
\]
\item $(x,y)$ is called a {\em proximal pair} for $f$ if
\[
\liminf_{i\to\infty}d(f^i(x),f^i(y))=0,
\]
\item $(x,y)$ is called an {\em asymptotic pair} for $f$ if
\[
\limsup_{i\to\infty}d(f^i(x),f^i(y))=0.
\]
\item $(x,y)$ is called a {\em scrambled pair} for $f$ if
\[
\limsup_{i\to\infty}d(f^i(x),f^i(y))>0\:\text{ and }\:\liminf_{i\to\infty}d(f^i(x),f^i(y))=0.
\]
\end{itemize}
\end{defi}

For a continuous map $f\colon X\to X$ and $x\in X$, the {\em $\omega$-limit set} $\omega(x,f)$ of $x$ for $f$ is defined to be the set of $y\in X$ such that $\lim_{j\to\infty}f^{i_j}(x)=y$ for some sequence $0\le i_1<i_2<\cdots$. Note that $\omega(x,f)$ is a closed $f$-invariant subset of $X$ and satisfies $y\rightarrow z$ for all $y,z\in\omega(x,f)$. For every $x\in X$, we have $\omega(x,f)\subset C$ for some $C\in\mathcal{C}(f)$ and such $C$ satisfies $C\subset C(x)$.

\begin{rem}
\normalfont
Let $f\colon X\to X$ be a continuous map.
\begin{itemize}
\item For a closed $f$-invariant subset $S$ of $X$ and $e>0$, we say that $x\in S$ is an {\em $e$-sensitive point} for $f|_S\colon S\to S$ if for any $\epsilon>0$, there is $y\in S$ such that $d(x,y)\le\epsilon$ and $d(f^i(x),f^i(y))>e$ for some $i\ge0$. We define $Sen_e(f|_S)$ to be the set of $e$-sensitive points for $f|_S$ and
\[
Sen(f|_S)=\bigcup_{e>0}Sen_e(f|_S).
\]    
\item A closed $f$-invariant subset $M$ of $X$ is said to be a {\em minimal set} for $f$ if closed $f$-invariant subsets of $M$ are only $\emptyset$ and $M$. This is equivalent to $M=\omega(x,f)$ for all $x\in M$.
\end{itemize}
\end{rem}
 
In \cite{K2}, the author gave three sufficient conditions for a shadowable point to be an entropy point. The next theorem refines Corollary 1.1 of \cite{K2}.

\begin{thm}
Let $f\colon X\to X$ be a continuous map. For any $x\in X$ and $C\in\mathcal{C}(f)$ with $\omega(x,f)\subset C$, if one of the following conditions is satisfied, then $C\in\mathcal{C}_{no}(f|_{C(x)})$.
\begin{itemize}
\item[(1)] $\omega(x,f)\cap Sen(f|_{CR(f)})\ne\emptyset$,
\item[(2)] there is $y\in X$ such that $(x,y)\in X^2$ is a scrambled pair for $f$,
\item[(3)] $\omega(x,f)$ is not a minimal set for $f$.
\end{itemize}
\end{thm}

\begin{rem}
\normalfont
For a continuous map $f\colon X\to X$, $y\in X$ is called a {\em minimal point} for $f$ if $y\in\omega(y,f)$ and $\omega(y,f)$ is a minimal set for $f$. Due to Theorem 8.7 of \cite{F}, we know that for any $x\in X$, there is a minimal point $y\in X$ for $f$ such that $(x,y)$ is a proximal pair for $f$. If $\omega(x,f)$ is not a minimal set for $f$, then it follows that $(x,y)$ is a scrambled pair for $f$, thus $(3)$ always implies $(2)$.
\end{rem}

We consider another local property of dynamical systems so-called h-expansiveness \cite{B1}. Let $f\colon X\to X$ be a continuous map. For $x\in X$ and $\epsilon>0$, let
\[
\Phi_{\epsilon}(x)=\{y\in X\colon d(f^i(x),f^i(y))\le\epsilon\:\:\text{for all $i\ge0$}\}
\]
and
\[
h_f^\ast(\epsilon)=\sup_{x\in X}h(f,\Phi_{\epsilon}(x)).
\]
We say that $f$ is {\em h-expansive} if $h_f^\ast(\epsilon)=0$ for some $\epsilon>0$. The following theorem gives several conditions equivalent to (non) h-expansiveness under the assumption of shadowing and chain transitivity.
  
\begin{thm}
Let $f\colon X\to X$ be a continuous map. If $f$ is chain transitive and has the shadowing property, then the following conditions are equivalent
\begin{itemize}
\item[(1)] $f$ is not h-expansive,
\item[(2)] for any $\epsilon>0$, there is $r>0$ such that for every $\delta>0$, there is a pair
\[
((x_i)_{i=0}^k,(y_i)_{i=0}^k)
\]
of $\delta$-chains of $f$ with $(x_0,x_k)=(y_0,y_k)$ and
\[
r\le\max_{0\le i\le k}d(x_i,y_i)\le\epsilon,
\]
\item[(3)] for any $\epsilon>0$, there are $m\ge1$ and a closed $f^m$-invariant subset $Y$ of $X$ such that
\[
\sup_{i\ge0}d(f^i(x),f^i(y))\le\epsilon
\]
for all $x,y\in Y$ and there is a factor map
\[
\pi\colon(Y,f^m)\to(\{0,1\}^\mathbb{N},\sigma),
\]
where $\sigma\colon\{0,1\}^\mathbb{N}\to\{0,1\}^\mathbb{N}$ is the shift map.
\item[(4)] for any $\epsilon>0$,  there is a scrambled pair $(x,y)\in X^2$ for $f$ such that
\[
\sup_{i\ge0}d(f^i(x),f^i(y))\le\epsilon.
\]
\end{itemize}
\end{thm}

In Section 4, we use this theorem to obtain a counter-example for a question in \cite{ACCV2}. We shall make some definitions to precisely state the properties that are satisfied by the example.

\begin{defi}
\normalfont
Let $f\colon X\to X$ be a continuous map and let $\xi=(x_i)_{i\ge0}$ be a sequence of points in $X$. For $\delta>0$, $\xi$ is called a {\em $\delta$-limit-pseudo orbit} of $f$ if $d(f(x_i),x_{i+1})\le\delta$ for all $i\ge0$, and
\[
\lim_{i\to\infty}d(f(x_i),x_{i+1})=0.
\]
For $\epsilon>0$, $\xi$ is said to be {\em $\epsilon$-limit shadowed} by $x\in X$ if $d(f^i(x),x_i)\leq \epsilon$ for all $i\ge 0$, and
\[
\lim_{i\to\infty}d(f^i(x),x_i)=0.
\]
We say that $f$ has the {\em s-limit shadowing property} if for any $\epsilon>0$, there is $\delta>0$ such that every $\delta$-limit-pseudo orbit of $f$ is $\epsilon$-limit shadowed by some point of $X$.
\end{defi}

\begin{rem}
\normalfont
If $f$ has the s-limit shadowing property, then $f$ satisfies the shadowing property.  
\end{rem}

\begin{defi}
\normalfont
Let $f\colon X\to X$ be a homeomorphism. For $x\in X$ and $\epsilon>0$, let
\[
\Gamma_\epsilon(x)=\{y\in X\colon d(f^i(x),f^i(y))\le\epsilon\:\:\text{for all $i\in\mathbb{Z}$}\}.
\]
We say that $f$ is 
\begin{itemize}
\item {\em expansive} if there is $e>0$ such that $\Gamma_e(x)=\{x\}$ for all $x\in X$,
\item {\em countably-expansive} if there is $e>0$ such that $\Gamma_e(x)$ is a countable set for all $x\in X$,
\item {\em cw-expansive} if there is $e>0$ such that $\Gamma_e(x)$ is totally disconnected for all $x\in X$.
\end{itemize}
\end{defi}

A continuous map $f\colon X\to X$ is said to be {\em transitive} (resp.\:{\em mixing}) if for any non-empty open subsets $U,V$ of $X$, it holds that $f^j(U)\cap V\ne\emptyset$ for some $j>0$ (resp.\:for all $j\ge i$ for some $i>0$).

\begin{rem}
\normalfont
If $f$ is transitive, then $f$ is chain transitive, and the converse holds when $f$ has the shadowing property.
\end{rem}

In Section 4, by using Theorem 1.4, we give an example of a homeomorphism $f\colon X\to X$ (Example 4.1) such that
\begin{itemize}
\item[(1)] $X$ is totally disconnected and so $f$ is cw-expansive,
\item[(2)] $f$ is mixing,
\item[(3)] $f$ is h-expansive,
\item[(4)] $f$ has the s-limit shadowing property,
\item[(5)] $f$ is not countably-expansive,
\item[(6)] $f$ satisfies $X_e=\emptyset$, where
\[
X_e=\{x\in X\colon\Gamma_\epsilon(x)=\{x\}\:\text{ for some $\epsilon>0$}\}.
\]
\end{itemize}

In \cite{ACCV1}, it is proved that if a homeomorphism $f\colon X\to X$ has the {\em L-shadowing property}, that is, a kind of two-sided s-limit shadowing property, then
\[
f|_{CR(f)}\colon CR(f)\to CR(f)
\]
is expansive if and only if $f|_{CR(f)}$ is countably-expansive if and only if $f|_{CR(f)}$ is h-expansive (see Corollary C of \cite{ACCV1}). Example 4.1 shows that even if a homeomorphism $f\colon X\to X$ satisfies the s-limit shadowing property, this equivalence does not hold. The example also gives a negative answer to the following question in \cite{ACCV2} (see Question 3 of \cite{ACCV2}):

\begin{ques}
Is $X_e$ non-empty for every transitive h-expansive and cw-expansive homeomorphism $f\colon X\to X$ satisfying the shadowing property?
\end{ques}

This paper consists of four sections. In Section 2, we collect some definitions, notations, and facts that are used in this paper. In Section 3, we prove Lemmas 1.1 and 1.2; prove Theorems 1.1, 1.2, and 1.3; and give an example mentioned in Remark 1.5. In Section 4, we prove Theorem 1.4 and give an example mentioned above (Example 4.1) after proving some auxiliary lemmas.

\section{Preliminaries}

In this section, we briefly collect some definitions, notations, and facts that are used in this paper. 

\subsection{{\it Odometers, equicontinuity, and chain continuity}}

An {\em odometer} (also called an {\em adding machine}) is defined as follows. Given an increasing sequence $m=(m_k)_{k\ge1}$ of positive integers such that $m_1\ge1$ and $m_k$ divides $m_{k+1}$ for each $k=1,2,\dots$, we define
\begin{itemize}
\item $X(k)=\{0,1,\dots,m_k-1\}$ (with the discrete topology),
\item
\[
X_m=\{(x_k)_{k\ge1}\in\prod_{k\ge1}X(k)\colon x_k\equiv x_{k+1}\pmod{m_k}\:\text{ for all $k\ge1$}\},
\]
\item $g_m(x)_k=x_k+1\pmod{m_k}$ for all $x=(x_k)_{k\ge1}\in X_m$ and $k\ge1$.
\end{itemize}
We regard $X_m$ as a subspace of the product space $\prod_{k\ge1}X(k)$. The homeomorphism
\[
g_m\colon X_m\to X_m
\]
(or $(X_m,g_m)$) is called an odometer with the periodic structure $m$.

Let $f\colon X\to X$ be a continuous map and let $S$ be a closed $f$-invariant subset of $X$. We say that $f|_S\colon S\to S$ is
\begin{itemize}
\item {\em equicontinuous} if for every $\epsilon>0$, there is $\delta>0$ such that any $x,y\in S$ with $d(x,y)\le\delta$ satisfies
\[
\sup_{i\ge0}d(f^i(x),f^i(y))\le\epsilon,
\]
\item {\em chain continuous} if for every $\epsilon>0$, there is $\delta>0$ such that any $\delta$-pseudo orbits $(x_i)_{i\ge0}$ and $(y_i)_{i\ge0}$ of $f$ with $x_0=y_0$ satisfies
\[
\sup_{i\ge0}d(x_i,y_i)\le\epsilon.
\]
\end{itemize}

Recall that for a continuous map $f\colon X\to X$, $\mathcal{C}_o(f)$ is defined as the set of $C\in\mathcal{C}(f)$ such that $C$ is a periodic orbit or an odometer, that is, $(C,f|_C)$ is topologically conjugate to an odometer.

\begin{lem}
For a continuous map $f\colon X\to X$, the following conditions are equivalent
\begin{itemize}
\item[(1)] $\mathcal{C}(f)=\mathcal{C}_o(f)$,
\item[(2)] $f|_{CR(f)}\colon CR(f)\to CR(f)$ is an equicontinuous homeomorphism and $CR(f)$ is totally disconnected,
\item[(3)] $f|_{CR(f)}\colon CR(f)\to CR(f)$ is chain continuous.
\end{itemize}
\end{lem}

\begin{proof}
We prove the implication $(1)\implies (2)$. Since $\mathcal{C}(f)=\mathcal{C}_o(f)$,
\begin{itemize}
\item[(A)] every $C\in\mathcal{C}(f)$ is totally disconnected,
\item[(B)] $f|_{CR(f)}$ is a distal homeomorphism, that is, every $(x,y)\in CR(f)^2$ is a distal pair for $f|_{CR(f)}$.
\end{itemize}
Since the quotient space
\[
\mathcal{C}(f)=CR(f)/{\leftrightarrow}
\]
is totally disconnected, by (A), we obtain that $CR(f)$ is totally disconnected. By (B)
and Corollary 1.9 of \cite{AGW}, we conclude that $f|_{CR(f)}$ is an equicontinuous homeomorphism. For a proof of  $(2)\implies(3)$ (resp.\:$(3)\implies(1)$), we refer to Lemma 3.3 (resp.\:Section 6) of \cite{K3}.
\end{proof}

By applying Lemma 2.1 to $f|_C\colon C\to C$, $C\in\mathcal{C}(f)$, we obtain the following corollary.

\begin{cor}
For a continuous map $f\colon X\to X$ and $C\in\mathcal{C}(f)$, the following conditions are equivalent
\begin{itemize}
\item[(1)] $C\in\mathcal{C}_o(f)$,
\item[(2)] $f|_C\colon C\to C$ is an equicontinuous homeomorphism and $C$ is totally disconnected,
\item[(3)] $f|_C\colon C\to C$ is chain continuous.
\end{itemize}
\end{cor}

\subsection{{\it Factor maps and inverse limit}}

For two continuous maps $f\colon X\to X$, $g\colon Y\to Y$, where $X$, $Y$ are compact metric spaces, a continuous map $\pi\colon X\to Y$ is said to be a {\em factor map} if $\pi$ is surjective and satisfies $\pi\circ f=g\circ\pi$. A factor map $\pi\colon X\to Y$ is also denoted as 
\[
\pi\colon(X,f)\to(Y,g).
\]

Given an inverse sequence of factor maps
\[
\pi=(\pi_n\colon(X_{n+1},f_{n+1})\to(X_n,f_n))_{n\ge1},
\]
let
\[
X=\{x=(x_n)_{n\ge1}\in\prod_{n\ge1}X_n\colon\pi_n(x_{n+1})=x_n\:\text{ for all $n\ge1$}\},
\]
which is a compact metric space. Then, a continuous map $f\colon X\to X$ is well-defined by $f(x)=(f_n(x_n))_{n\ge1}$ for all $x=(x_n)_{n\ge1}\in X$.
We call
\[
(X,f)=\lim_\pi(X_n,f_n)
\]
the {\em inverse limit system}. It is easy to see that $f$ is transitive (resp.\:mixing) if and only if  $f_n\colon X_n\to X_n$ is transitive (resp.\:mixing) for all $n\ge1$.  It is also easy to see that $f$ has the shadowing property if $f_n\colon X_n\to X_n$ has the shadowing property for all $n\ge1$.
 
\section{Proofs of Theorems 1.1, 1.2, and 1.3} 

In this section, we prove Lemmas 1.1 and 1.2; prove Theorems 1.1, 1.2, and 1.3; and give an example mentioned in Remark 1.5.

First, we prove Lemma 1.1.

\begin{proof}[Proof of Lemma 1.1]
If $C(x)$ is not chain stable, then there is $r>0$ such that for any $\delta>0$, there is a $\delta$-chain $x^{(\delta)}=(x_i^{(\delta)})_{i=0}^{k_\delta}$ of $f$ with $x_0^{(\delta)}\in C(x)$ and
\[
d(x_{k_\delta}^{(\delta)}, C(x))\ge r.
\]
Then, there are a sequence $0<\delta_1>\delta_2>\cdots$ and $y,z\in X$ such that the following conditions are satisfied
\begin{itemize}
\item $\lim_{j\to\infty}\delta_j=0$,
\item $\lim_{j\to\infty}x_0^{(\delta_j)}=y$ and $\lim_{j\to\infty}x_{k_{\delta_j}}^{(\delta_j)}=z$.
\end{itemize}
It follows that $y\in C(x)$, $d(z,C(x))\ge r>0$ and so $z\not\in C(x)$; and $y\rightarrow z$. However, if $y=x$, we obtain $x\rightarrow z$ implying $z\in C(x)$, a contradiction. If $y\ne x$, by $x\rightarrow y$ and $y\rightarrow z$, we obtain $x\rightarrow z$ implying $z\in C(x)$, a contradiction. Thus, the lemma has been proved.
\end{proof}

Next, we prove Lemma 1.2.

\begin{proof}[Proof of Lemma 1.2]
We prove the implication $(1)\implies(2)$. Let $x\in Sh(f)$ and $y\in C(x)\setminus\{x\}$. For any $\epsilon>0$, since $x\in Sh(f)$, there is $\delta>0$ such that every $\delta$-pseudo orbit $(x_i)_{i\ge0}$ of $f$ with $x_0=x$ is $\epsilon$-shadowed by some $z\in X$. Since $y\in C(x)\setminus\{x\}$ and so $x\rightarrow y$, we have a $\delta$-chain $\alpha=(y_i)_{i=0}^k$ of $f$ with $y_0=x$ and $y_k=y$. For any $\delta$-pseudo orbit $\beta=(z_i)_{i\ge0}$ of $f$ with $z_0=y$, we consider a $\delta$-pseudo orbit
\[
\xi=\alpha\beta=(x_i)_{i\ge0}=(y_0,y_1,\dots,y_{k-1},z_0,z_1,z_2,\dots)
\]
of $f$. Then, since $x_0=y_0=x$, $\xi$ is $\epsilon$-shadowed by some $z\in X$ and so $\beta$ is $\epsilon$-shadowed by $f^k(z)$. Since $\epsilon>0$ is arbitrary, we obtain $y\in Sh(f)$, thus $(1)\implies(2)$ has been proved.

Next, we prove the implication $(2)\implies(3)$. For a closed subset $K$ of $X$, if $K\subset Sh(f)$, then by Lemma 2.4 of \cite{K1}, for any $\epsilon>0$, there is $\delta>0$ such that every $\delta$-pseudo orbit $(x_i)_{i\ge0}$ of $f$ with $x_0\in K$ is $\epsilon$-shadowed by some $y\in X$. Since $C(x)$ is a closed subset of $X$, this clearly implies that if $C(x)\subset Sh(f)$, then $f$ has the shadowing on $C(x)$.

Finally, we  prove the implication $(3)\implies(1)$. If $f$ has the shadowing on $C(x)$, then for any $\epsilon>0$, there is $\delta>0$ such that every $\delta$-pseudo orbit $(x_i)_{i\ge0}$ of $f$ with $x_i\in C(x)$ for all $i\ge0$ is $\epsilon$-shadowed by some $y\in X$. Since $x\in C(x)$ and $C(x)$ is chain stable, if $\gamma>0$ is sufficiently small, then for every $\gamma$-pseudo orbit $\xi=(y_i)_{i\ge0}$ of $f$ with $y_0=x$, by taking $x_i\in C(x)$, $i>0$, with $d(y_i,C(x))=d(y_i,x_i)$ for all $i>0$, we have that
\begin{itemize}
\item $d(x_i,y_i)\le\epsilon$ for each $i>0$,
\item
\[
(x_i)_{i\ge0}=(x,x_1,x_2,x_3,\dots)
\]
is a $\delta$-pseudo orbit of $f$ with $x_i\in C(x)$ for all $i\ge0$ and so is $\epsilon$-shadowed by some $y\in X$.
\end{itemize}
It follows that $\xi$ is $2\epsilon$-shadowed by $y$. Since $\epsilon>0$ is arbitrary, we obtain $x\in Sh(f)$, thus $(3)\implies(1)$ has been proved. This completes the proof of Lemma 1.2. 
\end{proof}

We give a proof of Theorem 1.1.

\begin{proof}[Proof of Theorem 1.1]
First, we prove the ``if" part. Let $C\in\mathcal{C}_{no}(f|_{C(x)})$. Due to Corollary 2.1, since $f|_C\colon C\to C$ is not chain continuous, there are $p\in C$ and $e>0$ such that for any $\delta>0$, there are $\delta$-chains $(x_i)_{i=0}^k$ and $(y_i)_{i=0}^k$ of $f|_C$ with $x_0=y_0=p$ and $d(x_k,y_k)>e$. Fix $0<r<e$ and take any  $\epsilon>0$ with $r+2\epsilon<e$. Since $x\in Sh(f)$, there is $\delta_0>0$ such that every $\delta_0$-pseudo orbit $(x_i)_{i\ge0}$ of $f$ with $x_0=x$ is $\epsilon$-shadowed by some $y\in X$. We fix a pair
\[
((x_i)_{i=0}^K,(y_i)_{i=0}^K)
\]
of $\delta_0$-chains $f|_C$ with $x_0=y_0=p$ and $d(x_K,y_K)>e$. Since $C\subset C(x)$, we have $x\rightarrow q$ for some $q\in C$. We also fix a $\delta_0$-chain $\alpha=(z_i)_{i=0}^L$ of $f$ with $z_0=x$ and $z_L=q$. Since $f|_C$ is chain transitive, by compactness of $C$, there is $M>0$ such that for any $w\in C$, there is a $\delta_0$-chain $(w_i)_{i=0}^m$ of $f|_C$ with $w_0=w$, $w_m=p$, and $m\le M$. It follows that for any $w\in C$, there is a pair
\[
(a^w,b^w)=((a_i^w)_{i=0}^{k_w},(b_i^w)_{i=0}^{k_w})
\]
of $\delta_0$-chains of $f|_C$ with $a_0^w=b_0^w=w$, $d(a_{k_w}^w,b_{k_w}^w)>e$, and $k_w\le K+M$. Given any  $N\ge1$ and $s=(s_i)_{i=1}^N\in\{a,b\}^N$, we inductively define a family of $\delta_0$-chains
\[
\alpha(s,n)=(c(s,n)_i)_{i=0}^{k(s,n)}
\]
of $f|_C$, $1\le n\le N$, by $\alpha(s,1)=s_1^q$ and $\alpha(s,n+1)=s_{n+1}^{c(s,n)_{k(s,n)}}$ for any $1\le n\le N-1$. Then, we consider a family of $\delta_0$-chains
\[
\alpha(s)=(c(s)_i)_{i=0}^{k(s)}=\alpha\alpha(s,1)\alpha(s,2)\cdots\alpha(s,N)
\]
of $f$, $s\in\{a,b\}^N$. Note that $c(s)_0=x$ and $k(s)\le L+N(K+M)$ for all $s\in\{a,b\}^N$; and for any $s,t\in\{a,b\}^N$ with $s\ne t$, we have $d(c(s)_i,c(t)_i)>e$ for some
\[
0\le i \le\min\{k(s),k(t)\}\le L+N(K+M).
\]
By the choice of $\delta_0$, for every $s\in\{a,b\}^N$, there is $x(s)\in X$ such that $d(f^i(x(s)),c(s)_i)\le\epsilon$ for all $0\le i\le k(s)$. It follows that
\[
\{x(s)\colon s\in\{a,b\}^N\}
\]
is an $(L+N(K+M),r)$-separated subset of $B_\epsilon(x)=\{y\in X\colon d(x,y)\le\epsilon\}$. Since $N\ge1$ is arbitrary, we obtain
\begin{align*}
h(f,B_\epsilon(x),r)&=\limsup_{n\to\infty}\frac{1}{n}\log{s_n(f,B_\epsilon(x),r)}\\
&\ge\limsup_{N\to\infty}\frac{1}{L+N(K+M)}\log{s_{L+N(K+M)}(f,B_\epsilon(x),r)}\\
&\ge\limsup_{N\to\infty}\frac{1}{L+N(K+M)}\log{2^N}\\
&=\frac{1}{K+M}\log{2}>0.
\end{align*}
Since $\epsilon>0$ with $r+2\epsilon<e$ is arbitrary, we conclude that $x\in Ent_r(f)$, proving the ``if" part. 

Next, we prove the ``only if" part. Let $x\in Ent_r(f)$ for some $r>0$. Due to Lemma 2.1, it suffices to show that
\[
f|_{CR(f|_{C(x)})}\colon CR(f|_{C(x)})\to CR(f|_{C(x)})
\]
is not chain continuous. For any $\epsilon>0$, let
\[
S_\epsilon=\{y\in C(x)\colon d(y,CR(f|_{C(x)}))\le\epsilon\}
\]
and
\[
T_\epsilon=\{y\in C(x)\colon d(y,CR(f|_{C(x)}))\ge\epsilon\}.
\]
Since
\[
CR(f|_{C(x)})= C(x)\cap CR(f),
\]
we have $T_\epsilon\cap CR(f)=\emptyset$; therefore, for any $p\in T_\epsilon$, we can take a neighborhood $U_p$ of $p$ in $X$ such that
\begin{itemize}
\item[(1)] $d(a,b)\le r$ and $d(f(a),f(b))\le\epsilon$ for all $a,b\in U_p$,
\item[(2)] $f^i(c)\not\in U_p$ for all $c\in U_p$ and $i>0$.
\end{itemize}
We take $p_1,p_2,\dots,p_M\in T_\epsilon$ with $T_\epsilon\subset\bigcup_{j=1}^M U_{p_j}$. Let $U=\bigcup_{j=1}^M U_{p_j} $and take $0<\Delta\le\epsilon$ such that
\[
\{z\in X\colon d(z,T_\epsilon)\le\Delta\}\subset U.
\]
Since $x\in C(x)$ and $C(x)$ is chain stable, we can take a closed neighborhood $K$ of $x$ in $X$ such that
\begin{itemize}
\item[(3)] $d(a,b)\le\epsilon$ for all $a,b\in K$,
\item[(4)] $d(f^i(c),C(x))\le\Delta$ for all $c\in K$ and $i\ge0$.
\end{itemize}
For any $q\in X$ and $n\ge1$, let
\[
A(q,n)=\{0\le i\le n-1\colon f^i(q)\in U\}
\]
and take
\[
g(q,n)\colon A(q,n)\to\{U_{p_j}\colon1\le j\le M\}
\]
such that $f^i(q)\in g(q,n)(i)$ for every $i\in A(q,n)$. By (2), we have $|A(q,n)|\le M$ for all $q\in X$ and $n\ge1$. Note that
\[
|\{(A(q,n),g(q,n))\colon q\in X\}|\le\sum_{k=0}^{\min\{n,M\}}\binom{n}{k}M^k\le (M+1)n^M M^M.
\]
for all $n\ge1$. Since $x\in Ent_r(f)$, we have
\[
h(f,K,r)=\limsup_{n\to\infty}\frac{1}{n}\log{s_n(f,K,r)}>0;
\]
therefore,
\[
s_n(f,K,r)>(M+1)n^M M^M
\]
for some $n\ge1$. This implies that there are $u,v\in K$ such that $d_n(u,v)>r$ and
\[
(A(u,n),g(u,n))=(A(v,n),g(v,n)).
\]
We fix $0\le N\le n-1$ with $d(f^N(u),f^N(v))>r$ and let
\[
(A,g)=(A(u,n),g(u,n))=(A(v,n),g(v,n)).
\]
If
\[
A\cap\{0\le l\le N\}=\emptyset,
\]
then $f^l(u),f^l(v)\not\in U$ for all $0\le l\le N$. By $(4)$ and the choice of $\Delta$, for any $0\le l\le N$, we obtain
\[
\{f^l(u),f^l(v)\}\subset\{w\in X\colon d(w,S_\epsilon)\le\Delta\}.
\]
It follows that
\[
\max\{d(f^l(u),CR(f|_{C(x)})),d(f^l(v),CR(f|_{C(x)}))\}\le\epsilon+\Delta\le2\epsilon
\]
for all $0\le l\le N$. Moreover, by $u,v\in K$  and (3), we obtain $d(u,v)\le\epsilon$. If
\[
A\cap\{0\le l\le N\}\ne\emptyset,
\]
letting
\[
L=\max\:[A\cap \{0\le l\le N\}],
\]
we have $f^L(u),f^L(v)\in g(L)$ and $g(L)\in\{U_{p_j}\colon1\le j\le M\}$. By (1), we have $L<N$ and $d(f^{L+1}(u),f^{L+1}(v))\le\epsilon$. By
\[
A\cap \{L+1\le l\le N\}=\emptyset,
\]
$(4)$, and the choice of $\Delta$, similarly as above, we obtain
\[
\max\{d(f^l(u),CR(f|_{C(x)})),d(f^l(v),CR(f|_{C(x)}))\}\le\epsilon+\Delta\le2\epsilon
\]
for all $L+1\le l\le N$. Since $\epsilon>0$ is arbitrary, we conclude that $f|_{CR(f|_{C(x)})}$ is not chain continuous, thus the ``only if" part has been proved. This completes the proof of Theorem 1.1.
\end{proof} 

For the proof of Theorem 1.2, we need two lemmas.

\begin{lem}
Let $f\colon X\to X$ be a continuous map.
\begin{itemize}
\item[(1)] For any $x,y\in X$ and $r>0$, if $x\in Sh(f)$ and $y\in C(x)\cap Ent_r(f)$, then $x\in Ent_s(f)$ for all $0<s<r$.
\item[(2)] For any  $x,y\in X$ and $r,b>0$, if $x\in Sh(f)$ and $y\in C(x)\cap Ent_{r,b}(f)$, then $x\in Ent_{s,b}(f)$ for all $0<s<r$.
\end{itemize}
\end{lem}

\begin{proof}
Let $x\in Sh(f)$ and $y\in C(x)\setminus\{x\}$. For any $0<s<r$, we fix $\epsilon>0$ with $s+2\epsilon<r$. Since $x\in Sh(f)$, there is $\delta>0$ such that every $\delta$-pseudo orbit of $(x_i)_{i\ge0}$ of $f$ with $x_0=x$ is $\epsilon$-shadowed by some $z\in X$. Since $y\in C(x)\setminus\{x\}$ and so $x\rightarrow y$, we have a $\delta/2$-chain $(y_i)_{i=0}^k$ of $f$ with $y_0=x$ and $y_k=y$. For $K\in\mathcal{K}(y)$, $n\ge1$, and $r>0$, we take an $(n,r)$-separated subset $E(K,n,r)$ of $K$ with $|E(K,n,r)|=s_n(f,K,r)$. If $K$ is sufficiently small, then for any $p\in E(K,n,r)$,
\[
(z_i^p)_{i=0}^{k+n-1}=(y_0,y_1,\dots,y_{k-1},p,f(p),\dots,f^{n-1}(p))
\]
is a $\delta$-chain of $f$ with $z_0^p=y_0=x$ and so there is $z_p\in X$ with $d(f^i(z_p),z_i^p)\le\epsilon$ for all $0\le i\le k+n-1$. It follows that
\[
\{z_p\colon p\in E(K,n,r)\}
\]
is a $(k+n,s)$-separated subset of $B_\epsilon(x)=\{w\in X\colon d(x,w)\le\epsilon\}$ and so
\[
s_{k+n}(f,B_\epsilon(x),s)\ge|E(K,n,r)|=s_n(f,K,r),
\]
implying
\begin{align*}
h(f,B_\epsilon(x),s)&=\limsup_{n\to\infty}\frac{1}{k+n}\log{s_{k+n}(f,B_\epsilon(x),s)}\\
&\ge\limsup_{n\to\infty}\frac{1}{k+n}\log{s_n(f,K,r)}\\
&=h(f,K,r).
\end{align*}
Since $\epsilon>0$ with $s+2\epsilon<r$ is arbitrary, if $y\in Ent_r(f)$ (resp.\:$y\in Ent_{r,b}(f)$ for some $b>0$), we obtain $x\in Ent_s(f)$ (resp.\:$x\in Ent_{s,b}(f)$). Since $0<s<r$ is arbitrary, the lemma has been proved.
\end{proof}

Let $f\colon X\to X$ be a continuous map. For $\delta,r>0$ and $n\ge1$, we say that two $\delta$-chains $(x_i)_{i=0}^n$ and $(y_i)_{i=0}^n$ of $f$ is {\em $(n,r)$-separated} if $d(x_i,y_i)>r$ for some $0\le i\le n$. Let
\[
s_n(f,X,r,\delta)
\]
denote the largest cardinality of a set of $(n,r)$-separated $\delta$-chains of $f$. The following lemma is from \cite{Mi}.

\begin{lem}[Misiurewicz]
\[
h_{\rm top}(f)=\lim_{r\to0}\lim_{\delta\to0}\limsup_{n\to\infty}\frac{1}{n}\log{s_n(f,X,r,\delta)}.
\]
\end{lem}

We give a proof of Theorem 1.2.

\begin{proof}[Proof of Theorem 1.2]
First, we prove the ``if" part. Since
\[
h_{\rm top}(f|_{C(x)})=h(f,C(x))>0,
\]
we have $h(f,C(x),r)>0$ for some $r>0$. Taking $0<b\le h(f,C(x),r)$, we obtain $C(x)\cap Ent_{r,b}(f)\ne\emptyset$. Since $x\in Sh(f)$, by Lemma 3.1, this implies $x\in Ent_{s,b}(f)$ for all $0<s<r$, thus the ``if" part has been proved.

Next, we prove the``only if" part. Let $x\in Ent_{r_0,b}(f)$ for some $r_0,b>0$. Then, we have
\[
h(f,K,r_0)=\limsup_{n\to\infty}\frac{1}{n}\log{s_n(f,K,r_0)}\ge b
\]
for all $K\in\mathcal{K}(x)$. Since $C(x)$ is chain stable, taking $0<s<r_0$, we obtain
\[
\limsup_{n\to\infty}\frac{1}{n}\log{s_n(f,C(x),s,\delta)}\ge b
\]
for all $\delta>0$. From Lemma 3.2, it follows that
\begin{align*}
h_{\rm top}(f|_{C_{(x)}})&=\lim_{r\to0}\lim_{\delta\to0}\limsup_{n\to\infty}\frac{1}{n}\log{s_n(f,C(x),r,\delta)}\\
&\ge\lim_{\delta\to0}\limsup_{n\to\infty}\frac{1}{n}\log{s_n(f,C(x),s,\delta)}\\
&\ge b>0,
\end{align*}
thus the ``only if" part has been proved. This completes the proof of Theorem 1.2.
\end{proof}
 
Next, we prove Theorem 1.3. The proof of the following lemma is left to the reader.
 
\begin{lem}
Let $f\colon X\to X$ be a continuous map and let $C\in\mathcal{C}(f)$. For any $\epsilon>0$, there is $\delta>0$ such that every $\delta$-chain $(x_i)_{i=0}^k$ of $f|_{CR(f)}$ with $x_0\in C$ satisfies $d(x_i,C)\le\epsilon$ for all $0\le i\le k$. 
\end{lem}

\begin{proof}[Proof of Theorem 1.3]
Due to Corollary 2.1, it is sufficient to show that each of the three conditions implies that $f|_C\colon C\to C$ is not chain continuous.
 
(1) Taking $y\in\omega(x,f)\cap Sen(f|_{CR(f)})$, we have
\[
y\in C\cap Sen_{e_0}(f|_{CR(f)})
\]
for some $e_0>0$. Taking $0<e<e_0$, we obtain that for any $\delta>0$, there are $\delta$-chains $(x_i)_{i=0}^k$ and $(y_i)_{i=0}^k$ of $f|_{CR(f)}$ with $x_0=y_0=x$ and $d(x_k,y_k)>e$. By Lemma 3.3, this implies that $f|_C$ is not chain continuous and thus $C\in\mathcal{C}_{no}(f|_{C(x)})$.

(2) Since $(x,y)$ is a scrambled pair for $f$, we have
\[
\liminf_{i\to\infty}d(f^i(x),f^i(y))=0;
\]
therefore, there are a sequence $0\le i_1<i_2<\cdots$ and $z\in X$ such that
\[
\lim_{j\to\infty}d(f^{i_j}(x),f^{i_j}(y))=0
\]
and
\[
\lim_{j\to\infty}f^{i_j}(x)=z,
\]
which implies $z\in\omega(x,f)\cap\omega(y,f)$ and so $\omega(y,f)\subset C$. By
\[
\omega(x,f)\cup\omega(y,f)\subset C,
\]
we obtain
\[
\lim_{i\to\infty}d(f^i(x),C)=\lim_{i\to\infty}d(f^i(y),C)=0.
\]
On the other hand, since $(x,y)$ is a scrambled pair for $f$, we have
\[
\limsup_{i\to\infty}d(f^i(x),f^i(y))>0.
\]
These condition clearly imply that $f|_C$ is not chain continuous and thus $C\in\mathcal{C}_{no}(f|_{C(x)})$.

(3) Since $\omega(x,f)$ is not a minimal set for $f$, we have a closed $f$-invariant subset $S$ of $\omega(x,f)$ such that
\[
\emptyset\ne S\ne\omega(x,f).
\]
Fix $p\in S$, $q\in\omega(x,f)$, and $e>0$ with $d(p,q)>e$. Since $\omega(x,f)\subset C$ and $f|_C\colon C\to C$ is chain transitive, for any $\delta>0$, there are $\delta$-chains $(x_i)_{i=0}^k$ and $(y_i)_{i=0}^k$ of $f|_C$ with $x_0=y_0=y_k=p$ and $x_k=q$. This implies that $f|_C$ is not chain continuous and thus $C\in\mathcal{C}_{no}(f|_{C(x)})$.
\end{proof}

Finally, we give an example mentioned in Remark 1.5. For a continuous map $f\colon X\to X$, $C\in\mathcal{C}(f)$ is said to be {\em terminal} if $C$ is chain stable. The proof of the following lemma is left to the reader.

\begin{lem}
Let $f\colon X\to X$ be a continuous map. For any $x\in X$ and $C\in\mathcal{C}(f)$ with $\omega(x,f)\subset C$, if $C$ is terminal, then
\[
C(x)=\{f^i(x)\colon i\ge0\}\cup C.
\]
\end{lem}

\begin{ex}
\normalfont
This example is taken from \cite{K4}. Let $\sigma\colon[-1,1]^\mathbb{N}\to[-1,1]^\mathbb{N}$ be the shift map and let $d$ be the metric on $[-1,1]^\mathbb{N}$ defined by
\[
d(x,y)=\sup_{i\ge1}2^{-i}|x_i-y_i|
\]
for all $x=(x_i)_{i\ge1},y=(y_i)_{i\ge1}\in[-1,1]^\mathbb{N}$.  Let $s=(s_k)_{k\ge 1}$ be a sequence of numbers with $1>s_1>s_2>\cdots$ and $\lim_{k\to\infty}s_k=0$. Put
\[
S=\{0\}\cup\{-s_k\colon k\ge1\}\cup\{s_k\colon k\ge1\},
\]
a closed subset of $[-1,1]$. We define a closed $\sigma$-invariant subset $X$ of $S^\mathbb{N}$ by
\[
X=\{x=(x_i)_{i\ge1}\in S^\mathbb{N}\colon|x_1|\ge|x_2|\ge\cdots\}.
\]
Let $f=\sigma|_X\colon X\to X$, $X_k=\{-s_k,s_k\}^\mathbb{N}$ for each $k\ge1$, and let $X_0=\{0^\infty\}$. Then, we have 
\[
CR(f)=\{x=(x_i)_{i\ge1}\in X\colon|x_1|=|x_2|=\cdots\}=X_0\cup\bigcup_{k\ge1}X_k
\]
and
\[
\mathcal{C}(f)=\{X_0\}\cup\{X_k\colon k\ge1\}.
\]
Note that $X_0$ is terminal. For a rapidly decreasing sequence $s=(s_k)_{k\ge 1}$, we can show that $f$ satisfies the shadowing property and so $X=Sh(f)$. Let $x=(s_1,s_2,s_3,\dots)$ and note that $x\in X$. Since $\omega(x,f)=X_0$ and $X_0$ is terminal, by Lemma 3.4, we obtain
\[
C(x)=\{f^i(x)\colon i\ge0\}\cup X_0.
\]
By Theorem 1.1, we see that
\[
x\not\in\bigcup_{r>0}Ent_r(f).
\]
We shall show that $x\in Ent(f)$. Let
\[
x_k=(s_1,s_2,\dots,s_k,s_k,s_k,\dots),
\]
$k\ge1$, and note that $x_k\in X$ for each $k\ge1$. For any $k\ge1$, since $h(f,X_k)\ge\log{2}>0$, we have
$h(f,X_k,r_k)>0$ and so $X_k\cap Ent_{r_k}(f)\ne\emptyset$ for some $r_k>0$. For every $k\ge1$, since $X_k\subset C(x_k)$, we obtain $C(x_k)\cap Ent_{r_k}(f)\ne\emptyset$; therefore, Lemma 3.1 implies that $x_k\in Ent_{s_k}(f)$ for all $0<s_k<r_k$. In particular, we have $x_k\in Ent(f)$ for all $k\ge1$. Since $\lim_{k\to\infty}x_k=x$, we conclude that $x\in Ent(f)$. 
\end{ex}

\section{Proof of Theorem 1.4 and an example}

In this section, we prove Theorem 1.4 and give an example mentioned in Section 1.

\begin{proof}[Proof of Theorem 1.4]
First, we prove the implication $(1)\implies(2)$. If $f$ is not h-expansive, then
\[
h_f^\ast(\epsilon)=\sup_{x\in X}h(f,\Phi_{\epsilon}(x))>0
\]
for all $\epsilon>0$. Given any $\epsilon>0$, we take $x\in X$ with $h(f,\Phi_{\epsilon/2}(x))>0$ and fix $r>0$ with $h(f,\Phi_{\epsilon/2}(x),r)>0$. For any $0<\Delta\le r$, we take an open cover $\mathcal{U}=\{U_i\colon 1\le i\le m\}$ of $X$ such that $d(a,b)\le\Delta$ for all $1\le i\le m$ and $a,b\in U_i$. Since
\[
h(f,\Phi_{\epsilon/2}(x),r)=\limsup_{n\to\infty}\frac{1}{n}\log{s_n(f,\Phi_{\epsilon/2}(x),r)}>0,
\]
we have
\[
s_n(f,\Phi_{\epsilon/2}(x),r)>m^2
\]
for some $n\ge1$. Then, there are $u,v\in\Phi_{\epsilon/2}(x)$ such that $d_n(u,v)>r$, $u,v\in U_{i_0}$, and $f^{n-1}(u),f^{n-1}(v)\in U_{i_{n-1}}$ for some $U_{i_0},U_{i_{n-1}}\in\mathcal{U}$. It follows that 
\[
\max\{d(u,v),d(f^{n-1}(u),f^{n-1}(v))\}\le\Delta\le r
\]
and 
\[
r<\max_{0\le i\le n-1}d(f^i(u),f^i(v))\le\epsilon.
\]
Since $0<\Delta\le r$ is arbitrary, this implies the existence of $r>0$ as in (2). Since $\epsilon>0$ is arbitrary, $(1)\implies(2)$ has been proved.

Next, we prove the implication $(2)\implies(3)$. The proof is similar to the proof of Lemma 3.1 in \cite{ACCV1}. Given any $\epsilon>0$, we choose $r>0$ as in (2). We fix $0<\gamma<\min\{\epsilon,r/2\}$. Since $f$ has the shadowing property, there is $\delta>0$ such that every $\delta$-pseudo orbit of $f$ is $\gamma$-shadowed by some point of $X$. By the choice of $r$, we obtain a pair
\[
(\alpha(0),\alpha(1))=((x_i)_{i=0}^k,(y_i)_{i=0}^k)
\]
of $\delta$-chains of $f$ with $(x_0,x_k)=(y_0,y_k)$ and
\[
r\le\max_{0\le i\le k}d(x_i,y_i)\le\epsilon.
\]
Then, the chain transitivity of $f$ gives a $\delta$-chain $\beta=(z_i)_{i=0}^l$ of $f$ with $z_0=x_k=y_k$ and $z_l=x_0=y_0$. For any $s=(s_n)_{n\ge1}\in\{0,1\}^\mathbb{N}$, we consider a $\delta$-pseudo orbit
\[
\Gamma(s)=\alpha(s_1)\beta\alpha(s_2)\beta\alpha(s_3)\beta\cdots
\]
of $f$. Let $m=k+l$,
\[
Y=\{y\in X\colon\text{$\Gamma(s)$ is $\gamma$-shadowed by $y$ for some $s\in\{0,1\}^\mathbb{N}$}\},
\]
and define a map $\pi\colon Y\to\{0,1\}^\mathbb{N}$ so that $\Gamma(\pi(y))$ is $\gamma$-shadowed by $y$ for all $y\in Y$. By a standard argument, we can show that the following conditions are satisfied
\begin{itemize}
\item $Y$ is a closed subset of $X$,
\item $f^{m}(Y)\subset Y$,
\item $\pi$ is well-defined,
\item $\pi$ is surjective,
\item $\pi$ is continuous,
\item $\pi\circ f^m=\sigma\circ\pi$, where $\sigma\colon\{0,1\}^\mathbb{N}\to\{0,1\}^\mathbb{N}$ is the shift map.
\end{itemize}
It follows that $Y$ is a closed $f^m$-invariant subset of $Y$ and
\[
\pi:(Y,f^m)\to(\{0,1\}^\mathbb{N},\sigma)
\]
is a factor map. By the definition of $\Gamma(s)$, $s\in\{0,1\}^\mathbb{N}$, we see that
\[
\sup_{i\ge0}d(f^i(x),f^i(y))\le3\epsilon
\]
for all $x,y\in Y$. Since $\epsilon>0$ is arbitrary, $(2)\implies(3)$ has been proved.

We shall prove the implication $(3)\implies(1)$. Given any $\epsilon>0$, we take $m\ge1$ and $Y$ as in (3). Take $p\in Y$ and note that $Y\subset\Phi_\epsilon(p)$. It follows that
\[
h_f^\ast(\epsilon)\ge h(f,\Phi_\epsilon(p))\ge h(f,Y)\ge\frac{1}{m}h(f^m,Y)\ge\frac{1}{m}h(\sigma,\{0,1\}^\mathbb{N})=\frac{1}{m}\log{2}>0.
\]
Since $\epsilon>0$ is arbitrary, $(3)\implies(1)$ has been proved.

The implication $(4)\implies(2)$ is obvious from the definitions. It remains to prove the implication $(3)\implies(4)$. Given any $\epsilon>0$, we take $m\ge1$ and $Y$ as in (3). Since
\[
h(f^m,Y)\ge h(\sigma,\{0,1\}^\mathbb{N})=\log{2}>0,
\]
by Corollary 2.4 of \cite{BGKM}, there is a scrambled pair $(x,y)\in Y^2$ for $f^m$. Then, $(x,y)$ is also a scrambled pair for $f$ and satisfies
\[
\sup_{i\ge0}d(f^i(x),f^i(y))\le\epsilon
\]
because $x,y\in Y$. Since $\epsilon>0$ is arbitrary, $(4)\implies(3)$ has been proved. This completes the proof of Theorem 1.4.
\end{proof}

We use Theorem 1.4 to obtain a counter-example for a question in \cite{ACCV2}. The example will be given as an inverse limit of the full-shift $(\{0,1\}^\mathbb{Z},\sigma)$ with respect to a factor map
\[
F\colon(\{0,1\}^\mathbb{Z},\sigma)\to(\{0,1\}^\mathbb{Z},\sigma).
\]

We need three auxiliary lemmas. A homeomorphism $f\colon X\to X$ is said to be expansive if there is $e>0$ such that
\[
\Gamma_e(x)=\{x\}
\]
for all $x\in X$ and such $e$ is called an {\em expansive constant} for $f$. It is known that for a homeomorphism $f\colon X\to X$ with an expansive constant $e>0$ and $x,y\in X$, if
\[
\sup_{i\ge0}d(f^i(x),f^i(y))\le e,
\]
then $(x,y)$ is an asymptotic pair for $f$. The following lemma gives a sufficient condition for an inverse limit system to be h-expansive.

\begin{lem}
Let $\pi=(\pi_n\colon(X_{n+1},f_{n+1})\to(X_n,f_n))_{n\ge1}$ be a sequence of factor maps such that for every $n\ge1$, $f_n\colon X_n\to X_n$ is an expansive transitive homeomorphism with the shadowing property. Let
\[
(Y,g)=\lim_\pi(X_n,f_n)
\]
and note that $g$ is a transitive homeomorphism with the shadowing property. If $g$ is not h-expansive, then for any $N\ge 1$, there are $M\ge N$ and a scrambled pair $(x_{M+1},y_{M+1})\in X_{M+1}^2$ for $f_{M+1}$ such that $(\pi_M(x_{M+1}),\pi_M(y_{M+1}))$ is an asymptotic pair for $f_M$.
\end{lem}

\begin{proof}
Let $D$ be a metric on $Y$. Let $d_n$ be a metric on $X_n$ and $e_n>0$ be an expansive constant for $f_n$ for each $n\ge1$. Given any $N\ge1$, we take $\epsilon_N>0$ such that for any $p=(p_n)_{n\ge1}$, $q=(q_n)_{n\ge1}\in Y$, $D(p,q)\le\epsilon_N$ implies $d_n(p_n,q_n)\le e_n$ for all $1\le n\le N$. Since $g$ is not h-expansive, by Theorem 1.4, there is a scrambled pair
\[
(x,y)=((x_n)_{n\ge1},(y_n)_{n\ge1})\in Y^2
\]
for $g$ with
\[
\sup_{i\ge0}D(g^i(x),g^i(y))\le\epsilon_N.
\]
Then, for every $1\le n\le N$, since 
\[
\sup_{i\ge0}d_n(f_n^i(x_n),f_n^i(y_n))=\sup_{i\ge0}d_n(g^i(x)_n,g^i(y)_n)\le e_n,
\]
$(x_n,y_n)$ is an asymptotic pair for $f_n$. Since $(x,y)$ is a scrambled pair for $g$ and so a proximal for $g$, $(x_n,y_n)$ is a proximal pair for $f_n$ for all $n\ge1$. If $(x_n,y_n)$ is an asymptotic pair for $f_n$ for all $n\ge1$, then $(x,y)$ is an asymptotic pair for $g$, which is a contradiction. Thus, there is $m\ge 1$ such that $(x_{m+1},y_{m+1})$ is a scrambled pair for $f_{m+1}$. Letting
\[
M=\min\{m\ge1\colon\text{$(x_{m+1},y_{m+1})$ is a scrambled pair for $f_{m+1}$}\},
\]
we see that $M\ge N$, $(x_{M+1},y_{M+1})$ is a scrambled pair for $f_{M+1}$, and $(x_M,y_M)$ is an asymptotic pair for $f_M$, thus the lemma has been proved.
\end{proof}

A map $F\colon X\to X$ is said to be an {\em open map} if for any open subset $U$ of $X$, $f(U)$ is an open subset of $X$. Any continuous open map $F\colon X\to X$ satisfies the following property: for every $r>0$, there is $\delta>0$ such that for any $s,t\in X$ with $d(s,t)\le\delta$ and $u\in F^{-1}(s)$, we have $d(u,v)\le r$ for some $v\in F^{-1}(t)$.

For a continuous map $f\colon X\to X$, a sequence $(x_i)_{i\ge0}$ of points in $X$ is called a {\em limit-pseudo orbit} of $f$ if
\[
\lim_{i\to\infty}d(f(x_i),x_{i+1})=0,
\]
and said to be {\em limit shadowed} by $x\in X$
if
\[
\lim_{i\to\infty}d(f^i(x),x_i)=0.
\]

The next lemma is needed for the proof of Lemma 4.3.

\begin{lem}
Let $f\colon X\to X$ be a homeomorphism and let $F\colon(X,f)\to(X,f)$ be a factor map such that
\begin{itemize}
\item[(1)] $F$ is an open map,
\item[(2)] $d(v,v')\ge1$ for all $t\in X$ and $v,v'\in F^{-1}(t)$ with $v\ne v'$.
\end{itemize}
Suppose that
\begin{itemize}
\item[(3)] $(x_i)_{i\ge0}$ is a limit-pseudo orbit of $f$ and limit-shadowed by $x\in X$,
\item[(4)] $(z_i)_{i\ge0}$ is a limit-pseudo orbit of $f$ with $z_i\in F^{-1}(x_i)$ for all $i\ge0$. 
\end{itemize}
Then, there is $z\in F^{-1}(x)$ such that $(z_i)_{i\ge0}$ is limit-shadowed by $z$. 
\end{lem}

\begin{proof}
By (3), letting $\delta_i=d(x_i,f^i(x))$, $i\ge0$, we have  $\lim_{i\to\infty}\delta_i=0$. By (1), we can take a sequence $r_i>0$, $i\ge0$, such that 
\begin{itemize}
\item $\lim_{i\to\infty}r_i=0$,
\item for any $i\ge0$, $s,t\in X$ with $d(s,t)\le\delta_i$, and $u\in F^{-1}(s)$, we have $d(u,v)\le r_i$ for some $v\in F^{-1}(t)$.
\end{itemize}
With use of (4), we fix $N\ge0$ satisfying the following conditions
\begin{itemize}
\item $0<r_i<1/2$ for all $i\ge N$,
\item $d(u,v)\le r_i$ implies $d(f(u),f(v))\le1/4$ for all $i\ge N$ and $u,v\in X$,
\item $d(f(z_i),z_{i+1})\le1/4$ for all $i\ge N$.
\end{itemize}
By $\delta_N=d(x_N,f^N(x))$ and $z_N\in F^{-1}(x_N)$, we obtain $w_N\in F^{-1}(f^N(x))$ with $d(z_N,w_N)\le r_N$. Note that
\[
F(f^j(w_N))=f^j(F(w_N))=f^j(f^N(x))=f^{N+j}(x)
\]
for every $j\ge0$. By induction on $j$, we prove that $d(z_{N+j},f^j(w_N))\le r_{N+j}$ for all $j\ge0$. Assume that $d(z_{N+j},f^j(w_N))\le r_{N+j}$ for some $j\ge0$. Then,
\begin{align*}
d(z_{N+j+1},f^{j+1}(w_N))&\le d(z_{N+j+1},f(z_{N+j}))+d(f(z_{N+j}),f(f^j(w_N)))\\
&\le 1/4+1/4=1/2.
\end{align*}
Since
\[
\delta_{N+j+1}=d(x_{N+j+1},f^{N+j+1}(x))
\]
and $z_{N+j+1}\in F^{-1}(x_{N+j+1})$, we have
\[
d(z_{N+j+1},w)\le r_{N+j+1}
\]
for some $w\in F^{-1}(f^{N+j+1}(x))$. Since $f^{j+1}(w_N)\in F^{-1}(f^{N+j+1}(x))$, by (2), we obtain
\begin{align*}
d(z_{N+j+1},w')&\ge d(f^{j+1}(w_N),w')-d(z_{N+j+1},f^{j+1}(w_N))\\
&\ge1-1/2=1/2>r_{N+j+1}
\end{align*}
for all $w'\in F^{-1}(f^{N+j+1}(x))$ with $w'\ne f^{j+1}(w_N)$. It follows that $w=f^{j+1}(w_N)$ and so
\[
d(z_{N+j+1},f^{j+1}(w_N))\le r_{N+j+1};
\]
therefore, the induction is complete. Let $z=f^{-N}(w_N)$ and note that
\[
f^N(F(z))=F(f^N(z))=F(w_N)=f^N(x).
\]
Since $f$ is a homeomorphism, we have $F(z)=x$, that is, $z\in F^{-1}(x)$. Moreover, we obtain
\[
\lim_{i\to\infty}d(z_i,f^i(z))=\lim_{j\to\infty}d(z_{N+j},f^{N+j}(z))=\lim_{j\to\infty}d(z_{N+j},f^j(w_N))=\lim_{j\to\infty}r_{N+j}=0,
\]
thus the lemma has been proved.
\end{proof}

The following lemma gives a sufficient condition for an inverse limit system to satisfy the s-limit shadowing property.   

\begin{lem}
Let $f\colon X\to X$ be a homeomorphism and let $F\colon(X,f)\to(X,f)$ be a factor map such that
\begin{itemize}
\item[(1)] $F$ is an open map,
\item[(2)] $d(v,v')\ge1$ for all $t\in X$ and $v,v'\in F^{-1}(t)$ with $v\ne v'$.
\end{itemize}
Let $(X_n,f_n)=(X,f)$ and $\pi_n=F\colon(X,f)\to(X,f)$ for all $n\ge1$. Let
\[
(Y,g)=\lim_\pi(X_n,f_n).
\]
If $f$ has the s-limit shadowing property, then $g$ satisfies the s-limit shadowing property.
\end{lem}

\begin{proof}
Let $D$ be a metric on $Y$. Given any $\epsilon>0$, we take $N\ge1$ and $\epsilon_N>0$ such that for any $p=(p_n)_{n\ge1}$, $q=(q_n)_{n\ge1}\in Y$, $d(p_N,q_N)\le\epsilon_N$ implies $D(p,q)\le\epsilon$. Since $f$ has the s-limit shadowing property, there is $\delta_N>0$ such that every $\delta_N$-limit-pseudo orbit of $f$ is $\epsilon_N$-limit shadowed by some point of $X$. We take $\delta>0$ such that 
$D(p,q)\le\delta$ implies $d(p_N,q_N)\le\delta_N$ for all $p=(p_n)_{n\ge1}$, $q=(q_n)_{n\ge1}\in Y$. Let $\xi=(x^{(i)})_{i\ge0}$ be a $\delta$-limit-pseudo orbit of $g$. Then, for every $i\ge0$, since $D(g(x^{(i)}),x^{(i+1)})\le\delta$, we have
\[
d(f(x_N^{(i)}),x_N^{(i+1)})=d(g(x^{(i)})_N,x_N^{(i+1)})\le\delta_N.
\]
Also, since $\lim_{i\to\infty}D(g(x^{(i)}),x^{(i+1)})=0$, we have
\[
\lim_{i\to\infty}d(f(x_n^{(i)}),x_n^{(i+1)})=\lim_{i\to\infty}d(g(x^{(i)})_n,x_n^{(i+1)})=0
\]
for all $n\ge1$. It follows that $(x_N^{(i)})_{i\ge0}$ is a $\delta_N$-limit-pseudo of $f$ and so $\epsilon_N$-limit shadowed by some $x_N\in X$. 
Then, since
\[
\lim_{i\to\infty}d(f(x_N^{(i)}),x_N^{(i+1)})=\lim_{i\to\infty}d(f^i(x_N),x_N^{(i)})=0
\]
and
\[
\lim_{i\to\infty}d(f(x_{N+1}^{(i)}),x_{N+1}^{(i+1)})=0,
\]
by Lemma 4.2, we have
\[
\lim_{i\to\infty}d(f^i(x_{N+1}),x_{N+1}^{(i)})=0
\]
for some $x_{N+1}\in F^{-1}(x_N)$. Inductively, we obtain $x_{N+k}\in X$, $k\ge0$, such that
\[
\lim_{i\to\infty}d(f^i(x_{N+k}),x_{N+k}^{(i)})=0
\]
and $x_{N+k+1}\in F^{-1}(x_{N+k})$ for all $k\ge0$. We define $y=(y_n)_{n\ge1}\in Y$ by
\begin{equation*}
y_n=
\begin{cases}
F^{N-n}(x_N)&\text{for all $1\le n\le N$}\\
x_n&\text{for all $n\ge N$}
\end{cases}
.
\end{equation*}
Given any $i\ge0$, by
\[
d(g^i(y)_N,x_N^{(i)})=d(f^i(y_N),x_N^{(i)})=d(f^i(x_N),x_N^{(i)})\le\epsilon_N,
\]
we obtain
\[
D(g^i(y),x^{(i)})\le\epsilon.
\]
Moreover, since
\begin{align*}
\lim_{i\to\infty}d(g^i(y)_n,x_n^{(i)})&=\lim_{i\to\infty}d(f^i(y_n),x_n^{(i)})\\
&=\lim_{i\to\infty}d(f^i(F^{N-n}(x_N)),x_n^{(i)})\\
&=\lim_{i\to\infty}d(F^{N-n}(f^i(x_N)),F^{N-n}(x_N^{(i)}))=0
\end{align*}
for all $1\le n\le N$; and
\begin{align*}
\lim_{i\to\infty}d(g^i(y)_{N+k},x_{N+k}^{(i)})&=\lim_{i\to\infty}d(f^i(y_{N+k}),x_{N+k}^{(i)})\\
&=\lim_{i\to\infty}d(f^i(x_{N+k}),x_{N+k}^{(i)})=0
\end{align*}
for all $k\ge0$, we obtain
\[
\lim_{i\to\infty}D(g^i(y),x^{(i)})=0.
\]
In other words, $\xi$ is $\epsilon$-limit shadowed by $y$. Since $\epsilon>0$ is arbitrary, we conclude that $g$ satisfies the s-limit shadowing property, completing the proof of the lemma.
\end{proof}

Finally, we give the example.

\begin{ex}
\normalfont
Let $\mathbb{Z}_2=\{0,1\}$. We define a metric $d$ on $\{0,1\}^\mathbb{Z}$ by
\[
d(x,y)=\sup_{n\in\mathbb{Z}}2^{-|n|}|x_n-y_n|
\]
for all $x=(x_n)_{n\in\mathbb{Z}}$, $y=(y_n)_{n\in\mathbb{Z}}\in\{0,1\}^\mathbb{Z}$. Note that the shift map
\[
\sigma\colon\{0,1\}^\mathbb{Z}\to\{0,1\}^\mathbb{Z}
\]
is an expansive mixing homeomorphism with the shadowing property and so satisfies the s-limit shadowing property (see, e.g.\:\cite{BGO}). We define a map $F\colon\{0,1\}^\mathbb{Z}\to\{0,1\}^\mathbb{Z}$ by for any $x=(x_n)_{n\in\mathbb{Z}}$, $y=(y_n)_{n\in\mathbb{Z}}\in\{0,1\}^\mathbb{Z}$, $y=F(x)$ if and only if 
\[
y_n=x_n+x_{n+1}
\]
for all $n\in\mathbb{Z}$. Note that $F$ gives a factor map 
\[
F\colon(\{0,1\}^\mathbb{Z},\sigma)\to(\{0,1\}^\mathbb{Z},\sigma).
\]

Given any $x=(x_n)_{n\in\mathbb{Z}},y=(y_n)_{n\in\mathbb{Z}},z=(z_n)_{n\in\mathbb{Z}},w=(w_n)_{n\in\mathbb{Z}}\in\{0,1\}^\mathbb{Z}$, assume that
\begin{itemize}
\item $(x,y)$ is an asymptotic pair for $\sigma$,
\item $(F(z),F(w))=(x,y)$.
\end{itemize}
Then, there is $N\ge0$ such that $x_n=y_n$ for all $n\ge N$. If $z_N=w_N$, we have $z_n=w_n$ for all $n\ge N$ and so $(z,w)$ is an asymptotic pair for $\sigma$. If $z_N\ne w_N$, we have $z_n\ne w_n$ for all $n\ge N$ and so
\[
\liminf_{i\to\infty}d(\sigma^i(z),\sigma^i(w))=1>0,
\]
thus $(z,w)$ is a distal pair for $\sigma$. In both cases, $(z,w)$ is not a scrambled pair for $\sigma$.

For any $m\ge1$ and $a=(a_n)_{n=-m}^m\in\{0,1\}^{2m+1}$, we define $b=(b_n)_{n=-m}^{m-1}\in\{0,1\}^{2m}$ by
\[
b_n=a_n+a_{n+1}
\]
for all $-m\le n\le m-1$. Letting
\[
S(a)=\{x=(x_n)_{n\in\mathbb{Z}}\colon x_n=a_n\:\text{ for all $-m\le n\le m$}\}
\]
and
\[
T(b)=\{x=(x_n)_{n\in\mathbb{Z}}\colon x_n=b_n\:\text{ for all $-m\le n\le m-1$}\},
\]
we obtain $F(S(a))=T(b)$, an open subset of $\{0,1\}^\mathbb{Z}$. Since $m\ge1$ and $a=(a_n)_{n=-m}^m\in\{0,1\}^{2m+1}$ are arbitrary, it follows that $F$ is an open map. Given any $y=(y_n)_{n\in\mathbb{Z}}\in\{0,1\}^\mathbb{Z}$, we define $\hat{y}=(\hat{y}_n)_{n\in\mathbb{Z}}\in\{0,1\}^\mathbb{Z}$ by $\hat{y}_n=y_n+1$ for all $n\in\mathbb{Z}$. Then, for any $x\in\{0,1\}^\mathbb{Z}$, taking $y\in F^{-1}(x)$, we have $F^{-1}(x)=\{y,\hat{y}\}$. Note that $d(y,\hat{y})=1$ for all $y\in\{0,1\}^\mathbb{Z}$.

Let $(X_n,f_n)=(\{0,1\}^\mathbb{Z},\sigma)$ and $\pi_n=F\colon(\{0,1\}^\mathbb{Z},\sigma)\to(\{0,1\}^\mathbb{Z},\sigma)$ for all $n\ge1$. Let
\[
(Y,g)=\lim_\pi(X_n,f_n)
\]
and let $D$ be a metric on $Y$. Since $\{0,1\}^\mathbb{Z}$ is totally disconnected and $\sigma\colon\{0,1\}^\mathbb{Z}\to\{0,1\}^\mathbb{Z} $ is a mixing homeomorphism,
\begin{itemize}
\item $Y$ is totally disconnected,
\item $g$ is a mixing homeomorphism.
\end{itemize}
By Lemmas 4.1 and 4.3, we obtain the following properties
\begin{itemize}
\item $g$ is h-expansive,
\item $g$ has the s-limit shadowing property.
\end{itemize}
We shall show that
\begin{itemize}
\item $g$ is not countably-expansive,
\item $g$ satisfies $Y_e=\emptyset$, where
\[
Y_e=\{q\in Y\colon\Gamma_\epsilon(q)=\{q\}\:\text{ for some $\epsilon>0$}\}.
\]
\end{itemize}
Let $q=(q_n)_{n\ge1}\in Y$ and $N\ge1$. Let $F^{-1}(x)=\{x^a,x^b\}$ for all $x\in\{0,1\}^\mathbb{Z}$. Then, for all $c=(c_k)_{k\ge1}\in\{a,b\}^\mathbb{N}$, we define $q(c)=(q(c)_n)_{n\ge1}\in Y$ by $q(c)_n=q_n$ for all $1\le n\le N$; and 
\[
q(c)_{N+k}=q(c)_{N+k-1}^{c_k}
\]
for all $k\ge 1$. Given any $\epsilon>0$, if $N\ge1$ is large enough, $q(c)$, $c\in\{a,b\}^\mathbb{N}$, satisfies $q(c)\in\Gamma_\epsilon(q)$ for all $c\in\{a,b\}^\mathbb{N}$. Since
\[
\{q(c)\colon c\in\{a,b\}^\mathbb{N}\}
\]
is an uncountable set, it follows that $g$  is not countably-expansive. Since $q\in Y$ and $\epsilon>0$ are arbitrary, it also follows that $Y_e=\emptyset$.
\end{ex}


\begin{thebibliography}{99}
\bibitem{A} D.V.\:Anosov, Geodesic flows on closed Riemann manifolds with negative curvature. Proc. Steklov Inst. Math. 90 (1967), 235 p.

\bibitem{AH} N.\:Aoki, K.\:Hiraide, Topological theory of dynamical systems. Recent advances. North--Holland Mathematical Library, 52. North--Holland Publishing Co., 1994.

\bibitem{AV} J.\:Aponte, H.\:Villavicencio, Shadowable points for flows. J. Dyn. Control Syst. 24 (2018), 701--719.

\bibitem{AR} A.\:Arbieto, E.\:Rego, Positive entropy through pointwise dynamics. Proc. Amer. Math. Soc. 148 (2020), 263--271.

\bibitem{ACCV1} A.\:Artigue, B.\: Carvalho, W.\:Cordeiro, J.\:Vieitez, Beyond topological hyperbolicity: the L-shadowing property. J. Differ. Equations 268 (2020), 3057--3080.

\bibitem{ACCV2} A.\:Artigue, B.\:Carvalho, W.\:Cordeiro, J.\:Vieitez, Countably and entropy expansive homeomorphisms with the shadowing property. Proc. Amer. Math. Soc. 150 (2022), 3369--3378.

\bibitem{AGW} J.\:Auslander, E.\:Glasner, B.\:Weiss, On recurrence in zero dimensional flows, Forum Math. 19 (2007), 107--114.

\bibitem{BGO} A.D.\:Barwell, C.\:Good, P.\:Oprocha, Shadowing and expansivity in subspaces. Fund. Math. 219 (2012), 223--243.

\bibitem{BGKM} F.\:Blanchard, E.\:Glasner, S.\:Kolyada, A.\:Maass, On Li-Yorke pairs. J. Reine Angew. Math. 547 (2002), 51--68.

\bibitem{B1} R.\:Bowen, Entropy-expansive maps. Trans. Amer. Math. Soc. 164 (1972), 323--331.

\bibitem{B2} R.\:Bowen, Equilibrium states and the ergodic theory of Anosov diffeomorphisms. Lecture Notes in Mathematics, 470. Springer--Verlag, 1975.

\bibitem{DKD} P.\:Das, A.G.\:Khan, T.\:Das, Measure expansivity and specification for pointwise dynamics. Bull. Braz. Math. Soc. (N.S.) 50 (2019), 933--948.

\bibitem{DJM} M.\:Dong, W.\:Jung, C.\:Morales, Eventually shadowable points. Qual. Theory Dyn. Syst. 19 (2020), 16, 11 pp.

\bibitem{DLM} M.\:Dong, K.\:Lee, C.\:Morales, Pointwise topological stability and persistence. J. Math. Anal. Appl. 480 (2019), 123334, 12 pp.

\bibitem{F} H.\:Furstenberg, Recurrence in ergodic theory and combinatorial number theory. Princeton University Press, 1981.

\bibitem{JLM} W.\:Jung, K.\:Lee, C.A.\:Morales, Pointwise persistence and shadowing. Monatsh. Math. 192 (2020), 111--123.

\bibitem{K1} N.\:Kawaguchi, Quantitative shadowable points. Dyn. Syst. 32 (2017), 504--518.

\bibitem{K2} N.\:Kawaguchi, Properties of shadowable points: chaos and equicontinuity. Bull. Braz. Math. Soc. (N.S.) 48 (2017), 599--622.

\bibitem{K3} N.\:Kawaguchi, Maximal chain continuous factor. Discrete Contin. Dyn. Syst. 41 (2021), 5915--5942.

\bibitem{K4} N.\:Kawaguchi, Generic and dense distributional chaos with shadowing. J. Difference Equ. Appl. 27 (2021), 1456--1481.

\bibitem{KD1} A.G.\:Khan, T.\:Das, Average shadowing and persistence in pointwise dynamics. Topology Appl. 292 (2021), 107629, 13 pp.

\bibitem{KD2} A.G.\:Khan, T.\:Das, Stability theorems in pointwise dynamics. Topology Appl. 320 (2022), 108218, 14 pp.

\bibitem{KDD} A.G.\:Khan, P.K.\:Das, T.\:Das, Pointwise dynamics under orbital convergence. Bull. Braz. Math. Soc. (N.S.) 51 (2020), 1001--1016.

\bibitem{KLM} N.\:Koo, K.\:Lee, C.A.\:Morales, Pointwise topological stability. Proc. Edinb. Math. Soc. 61 (2018), 1179--1191.

\bibitem{LNY} X.F.\:Luo, X.X.\:Nie, J.D.\:Yin, On the shadowing property and shadowable point of set-valued dynamical
systems. Acta Math. Sin. (Engl. Ser.), 1384--1394.

\bibitem{Mi} M.\:Misiurewicz, Remark on the definition of topological entropy. Dynamical systems and partial differential equations (Caracas, 1984), Univ. Simon Bolivar, Caracas, 1986, 65--67.

\bibitem{M} C.A.\:Morales, Shadowable points. Dyn. Syst. 31 (2016), 347--356.

\bibitem{P} S.Yu.\:Pilyugin, Shadowing in dynamical systems. Lecture Notes in Mathematics, 1706. Springer--Verlag, 1999.

\bibitem{RA} E.\:Rego, A.\:Arbieto, On the entropy of continuous flows with uniformly expansive points and the
globalness of shadowable points with gaps. Bull. Braz. Math. Soc. (N.S.) 53 (2022), 853--872.

\bibitem{S} B.\:Shin, On the set of shadowable measures. J. Math. Anal. Appl. 469 (2019), 872--881.

\bibitem{YZ} X.\:Ye, G.\:Zhang, Entropy points and applications. Trans. Amer. Math. Soc. 359 (2007), 6167--6186.
\end{thebibliography}
\end{document}